\RequirePackage{fix-cm}
\documentclass[smallextended]{svjour3}       
\smartqed  
\usepackage{graphicx}
\def\RR{I\!\!R}
\def\NN{I\!\! N}

\usepackage{amsmath}
\begin{document}

\title{On some $\overrightarrow{p(x)}$ anisotropic elliptic equations in unbounded domain
}


\author{A.  Aberqi         \and
        B. Aharrouch \and J. Bennouna 
}


\institute{A.  Aberqi \at
           Sidi Mohamed Ben Abdellah University; National School of Applied Sciences, Fez, Morocco    \\
              \email{aberqi$\_$ahmed@yahoo.fr}           
           \and
           B. Aharrouch \at
              Sidi Mohamed Ben Abdellah University, Faculty of Sciences Dhar El Mahraz
		Laboratory LAMA, Fez,Morocco\\
               \email{bnaliaharrouch@gmail.com} 
            \and
           J. Bennouna \at
           Sidi Mohamed Ben Abdellah University, Faculty of Sciences Dhar El Mahraz
		Laboratory LAMA, Fez,Morocco\\
           \email{jbennouna@hotmail.com} }

\date{Received: date / Accepted: date}

\maketitle

\begin{abstract}
We study a class of nonlinear elliptic problems with Dirichlet conditions in the framework of the Sobolev anisotropic spaces with variable exponent, involving an anisotropic operator on an unbounded domain $\Omega\subset \>I\!\!R^{N}\>(N \geq 2)\>$.\\
We prove the existence of entropy solutions avoiding sign condition and coercivity on the lowers order terms.
\end{abstract}

\section{Introduction}
Partial differential equations with non-standard growth have been widely studied in recent years, it thus finds applications in different fields of physics, image processing, filtration in porous media, optimal control and electrorheological fluids (smart fluids) which change their mechanical properties dramatically when
an external electric field is applied, for a model in the case of isothermal, homogeneous incompressible smart fluids see e.g. Rajagopal and Rusi$\widetilde{c}$ka (\cite{Rajagopal}), see also (\cite{Halsey}).\\

This paper concerns the following problem:
\begin{equation}
    (\mathcal{P}) \left\{\begin{array}{ll}
-\mbox{div}(a(x,u,\nabla u))+ H(x,u,\nabla u)+ |u|^{p_0(x)-2}u= f \quad \mbox{in}\quad \Omega,\\
u=0 \quad \mbox{on} \quad \quad \partial \Omega,\\
    \end{array}%
    \right.
    \end{equation}
 where $\displaystyle a: \Omega\times \RR\times\RR^{N}\to \RR^{N}$ with $a(x,s,\xi)=\Big(a_{1}(x,s,\xi),...,a_{N}(x,s,\xi)\Big)$, we assume that
 $a_{i}(x,s,\xi)$ for $i=1,...,N$ and $H(x,s,\xi)$ are Carath\'eodory functions satisfying assumptions ((\ref{eqa1})--(\ref{eqh})) below.\\
The exponent $p_0(x)$ is a measurable function defined on arbitrary domain $\Omega$; which satisfies (\ref{eq0}), the source $f$ is merely integrable.\\
$\bullet$ Starting by the case where the domain $\Omega $ is bounded in $\RR^{N}$, and consider the following problem:
\begin{equation}\label{mod1}
-\sum_{i=1}^{i=N}\partial_{x_{i}}\big(a_{i}(x)|u_{x_{i}}|^{p_{i}-2}u_{x_{i}}\big)- \sum_{i=1}^{i=N}\partial_{x_{i}}g_{i}(u)+|u|^{p_{0}-2}u=f(x),
\end{equation}
$x\in \Omega \subset \RR^{N}$, \,  $\Omega$ is a bounded domain.\\
The model (\ref{mod1}) is very well understood, in the isotropic case, i.e. $\displaystyle \overrightarrow{p}=(p_{1},...,p_{N})$\\
$= p\equiv \mbox{cte}$, for an lucid, yet precise comprehensive papers see (\cite{Boccardo}, \cite{Dall},\cite{Murrat},\cite{Porretta}).\\
For the anisotropic operator with polynomial growth i.e. $\displaystyle \overrightarrow{p}=(p_{1},p_{2},...,p_{N})$\\
$\, p_{i}\in \RR$ we mention the reference works of A. G. Korolev (\cite{Korolev}) and N. T. Chung et al. (\cite{Chung}), for more works in the classical anisotropic spaces $W^{1,\overrightarrow{p}}(\Omega)$ we refer the reader to (\cite{Antontsev},\cite{Boan}, \cite{Bouj}, \cite{B.A.B.A}, \cite{Chung1}, \cite{Di}, \cite{Fra}).\\
Now for the operators governed by non-standard growth, namely,\\
$\displaystyle \overrightarrow{p}(.)=(p_{1}(.),p_{2}(.),...,p_{N}(.))$ where $\displaystyle p_{i}:\Omega \to \RR$ are measurable functions, an excellent introduction is in (\cite{Fan}), another sources are (\cite{A.B.H.Y}, \cite{bendahmane2013nonlinear}, \cite{Boureanu}, \cite{Elhamdaoui}
\cite{Ouaro}).\\
$\bullet$ In the arbitrary domain $\Omega$, there are many studies which establish the existence of solutions in an unbounded domain, in particular, Bendahmanne and Karlsen (\cite{B.Karlsen}) proved the solvability and regularity of (\ref{mod1}) where $x$ lies in $\RR^{N}$, in the classical anisotropic spaces, (\cite{Bokalo}) and (\cite{Domanska}) solved (\ref{mod1}),
in the framework of anisotropic spaces with variable exponents, without lower order terms $g$ and perturbation $|u|^{p_{0}(x)-2}u$, they have shown the well-posedness
without constraint on the growth.\\
Our paper continues the work in this direction. We will show the existence of entropy solutions of (\ref{mod1}) with a general operator of type Leray-lions, and the presence of a lower order $\displaystyle g(x,u,\nabla u)$ which do not satisfy the sign condition, and a perturbation  $|u|^{p_{0}(x)-2}u$ in arbitrary domain of $\RR^{N}$.\\
Let us summarize the outline of this paper: In Section 2 we recall some basic notations and Sobolev inequality for anisotropic Sobolev spaces. Our main results are stated in Section 3, while the appendix in Section 4.
\section{Framework Space}
Let $\Omega$ be an unbounded domain of $\RR^{N}$, $N\geq 2$, we denote\\
$$ \mathcal{C}_{+}(\Omega)=\{ \mbox{measurable function} \,\,  p(\cdot): \Omega \to \RR, \, \mbox{such that} \, 1< p^{-} \leq p^{+} < \infty \}$$
where
$$ p^{-}= \mbox{essinf}\{p(x) \, / \, x\in \Omega\}\,\quad \quad   \mbox{and} \quad \quad\, p^{+}= \mbox{esssup}\{p(x) \, / \, x\in \Omega\}$$
In this section we present the anisotropic variable exponent Sobolev space, used in the study of the
 elliptic problem (1).\\
Let $p_{0}(x), p_{1}(x), ..., p_{N}(x) $ be $N+1$ variable exponents in $ \mathcal{C}_{+}(\Omega)$. We denote\\
$\displaystyle \overrightarrow{p}(\cdot)=\Big(p_{0}(\cdot), p_{1}(\cdot), ..., p_{N}(\cdot)\Big)\in \Big(\mathcal{C}_{+}(\Omega)\Big)^{N+1}$, and $\displaystyle \underline{p}=\min\Big(p^{-}_{0}, p^{-}_{1}, ..., p^{-}_{N}\Big)$ and $\displaystyle \overline{p}=\max\Big(p^{+}_{0}, p^{+}_{1},..., p^{+}_{N}\Big)$. Denoting by $L^{\overrightarrow{p}(\cdot)}(\Omega)$ the product space $\displaystyle \prod_{i=1}^N L^{p_{i}}(\Omega)$ endowed with the product norm
$\displaystyle ||u||_{\overrightarrow{p}(\cdot)}=\sum_{i=1}^N||u||_{p_{i}}$. We define the anisotropic variable exponent Sobolev space $W^{1,\overrightarrow{p}(\cdot)}(\Omega)$ as follow:
$$W^{1,\overrightarrow{p}(\cdot)}(\Omega)=\{u \in L^{p_{0}(\cdot)}(\Omega), \quad  \frac{\partial u}{\partial x_{i}} \in L^{p_{i}(\cdot)}(\Omega),\quad  \mbox{for} \quad \, i=1,...N\},$$
where $L^{p_{i}(\cdot)}(\Omega)$, for $i=0,...,N$ are the Lebesgue spaces with variable exponent $p_{i}(\cdot)$. We define also the space $\displaystyle W_{0}^{1,\overrightarrow{p}(\cdot)}(\Omega)$ as the closure of $\mathcal{C}_{0}^{\infty}(\Omega)$ in $W^{1,\overrightarrow{p}(\cdot)}(\Omega)$
with respect to the norm $\displaystyle||u||_{1,\overrightarrow{p}(\cdot)}=\sum_{i=0}^N\Big\|\frac{\partial u}{\partial x_{i}}\Big\|_{p_{i}(\cdot)}$.
The space $\displaystyle \Big(W_{0}^{1,\overrightarrow{p}(\cdot)}(\Omega), ||.||_{1,\overrightarrow{p}(\cdot)}\Big)$ is a Banach reflexive space see \cite{Fan}.\\
Let $\displaystyle \overline{p}(x)=N\Big(\sum_{i=1}^N\frac{1}{p_{i}(x)}\Big)^{-1}$,
 $\displaystyle p_{*}(x)=\left\{\begin{array}{lll}
\frac{N\overline{p}(x)}{N-\overline{p}(x)} & \mbox{if}&\overline{p}(x)\leq N,\\
+\infty & \mbox{if}& \overline{p}(x)\geq N.\\
    \end{array}%
    \right.$\\
    And $\displaystyle p_{\infty}=\max(p_{*}(x),p_{+}(x))\cdot$
\subsection{Basic Lemmas}
\begin{lemma}
Let $Q$ be a bounded domain of $\RR^{N}$ and $\vec{p}(\cdot)=(p_1(\cdot), p_2(\cdot),...,p_N(\cdot))\in (\mathcal{C}^+(\bar{Q}))^N.$ If $q(\cdot)\in\mathcal{C}^+(\bar{Q})$ and $$q(x)<p_{\infty}(x)~~\forall x\in Q,$$
then the embedding $W_0^{1,\vec{p}(\cdot)}(Q)\subset L^{q(\cdot)}(Q)$ is continuous and compact.
\end{lemma}
\begin{proof}
The previous embedding theorem for the space $W_{0}^{1,\overrightarrow{p}(\cdot)}(\Omega)$ is proved in Theorem 2.5 of \cite{Fan}.
\end{proof}
\begin{lemma}\label{cv gr}
Let the assumptions (\ref{eqa1} )--(\ref{eqa3}) be satisfied in $Q$, and for some fixed $k > 0$ there hold
\begin{equation}
u^j\rightharpoonup u~~ \mbox{in}~~L^{\vec{p}(\cdot)}(Q), ~~j\rightarrow\infty,
\end{equation}
\begin{equation}
T_k(u^j)\rightarrow ~~T_k(u)~~\mbox{a.e. in}~~Q, ~~j\rightarrow\infty,
\end{equation}
\begin{equation}
\lim_{j\rightarrow\infty}\int_{Q}(a(x,T_k(u^j),\nabla u^j)-a(x,T_k(u^j),\nabla u)).(\nabla u^j-\nabla u)=0.
\end{equation}
Then along a subsequence,
\begin{equation}\label{a.e. in}
\nabla u^j\rightarrow ~~\nabla u~~\mbox{a.e. in}~~Q, ~~j\rightarrow\infty,
\end{equation}
and
\begin{equation}
\nabla u^j\rightarrow ~~\nabla u~~\mbox{strongly}~~L^{\vec{p}(\cdot)}(Q), ~~j\rightarrow\infty.
\end{equation}
\end{lemma}
\begin{proof}
The convergence (\ref{a.e. in}) is established analogously as in the proof of \big[\cite{K.K.M}, Assertion 2\big].
Apparently, the first statement of this kind is Lemma 3.3 from the work \cite{Lions}.
\end{proof}
\begin{lemma}\label{cv faible}
Let $\Omega \subset \RR^{N}$ be an unbounded domain,  $(u_{n})_{n\in \NN}$ and $u$ be functions from $L^{\overrightarrow{p}(\cdot)}(\Omega)$, such that $(u_{n})_{n\in \NN}$ is bounded in $L^{\overrightarrow{p}(\cdot)}(\Omega)$ and \\
$u_{n}\rightarrow u \quad \mbox{a.e. in}\quad \Omega.$ Then
$$u_{n}\rightharpoonup u \quad \mbox{weakly in }\quad L^{\overrightarrow{p}(\cdot)}(\Omega).$$
\end{lemma}
\section{Assumptions and main result}
Consider $\Omega$ be an unbounded domain in $\RR^{N}$, $(N\geq 2)$,  $\displaystyle \overrightarrow{p}(\cdot)=\Big(p_{0}(\cdot),..., p_{N}(\cdot)\Big)$\\
$\in \Big(\mathcal{C}_{+}(\Omega)\Big)^{N+1}$, we will assume that
\begin{equation}\label{eq0}
\displaystyle p_{0}(x)\geq \underline{p}(x), \quad  x \in \Omega.
\end{equation}
Our aim in this work, is to prove an existence result for the following Dirichlet elliptic equations
\begin{equation}
    (\mathcal{P}) \left\{\begin{array}{ll}
-\mbox{div}(a(x,u,\nabla u)+ H(x,u,\nabla u)+ \mid u\mid^{p_0(x)-2}u= f \quad \mbox{in}\quad \Omega,\\
u=0 \quad \mbox{on} \quad \quad \partial \Omega.\\
    \end{array}%
    \right.
    \end{equation}
 Where $\displaystyle a: \Omega\times \RR\times\RR^{N}\to \RR^{N}$ with $a(x,s,\xi)=\Big(a_{1}(x,s,\xi),...,a_{N}(x,s,\xi)\Big)$, we assume that
 $a_{i}(x,s,\xi)$ for $i=1,...,N$ are Carath\'eodory functions, such that,
 for a.e. $x\in \Omega$ and for all $s\in \RR$, $\xi,\xi^{*}\in \RR^{N}$, $\xi\neq\xi^{*},$
 \begin{equation}\label{eqa1}
|a_i(x,s,\xi)|\leq  (\hat{a}_{i}(|s|)\Big((\sum_{i=1}^N|\xi_{i}|^{p_{i}(x)})^{\frac{1}{p'_{i}(x)}}+c_{i}(x)\Big),  \quad \mbox{with} \quad i=1,....N,
\end{equation}
\begin{equation}\label{eqa2}
\hspace{-4.5cm}(a_i(x,s,\xi)-a_i(x,s,\xi^{\ast})(\xi_i-\xi^{\ast}_i)>0,
\end{equation}
\begin{equation}\label{eqa3}
\hspace{-6cm}a_i(x,s,\xi).\xi_i\geq\alpha \sum_{i=1}^N|\xi_{i}|^{p_{i}(x)},
\end{equation}
with  $\displaystyle \hat{a}_{i}:\RR^{+}\to\RR^{+*}$. Furthermore, we assume the nonlinear term $\displaystyle H(x,s,\xi)$ is a Carath\'eodory function which satisfies only the growth condition:
\begin{equation}\label{eqh}
|H(x,s,\xi)|\leq \hat{h}(|s|)\sum_{i=1}^N|\xi_{i}|^{p'_{i}(x)} + h_{0}(x),\quad  \mbox{with}\quad  h_{0}\in L^{1}(\Omega),
\end{equation}
and $\displaystyle \hat{h}: \RR^+ \to \RR^{+*} $.
The source data $\displaystyle f \in L^{1}(\Omega)$. $T_k,$ $k>0,$ denotes the truncation function at level $k$ defined on $\RR$ by
 $\displaystyle T_k(r)=\max(-k,\min(k,r)).$\\
 We set:\\
$\mathcal{T}_{0}^{1,\overrightarrow{p}(\cdot)}(\Omega)=\{u:\Omega \to \RR,\mbox{measurable},\,T_{k}(u)\in W_{0}^{1,\overrightarrow{p}(\cdot)}(\Omega), \quad \forall k>0\}.$
\begin{definition}
A measurable function $u$ is said be an entropy solution for the problem $(\mathcal{P})$, if $\displaystyle u \in \mathcal{T}_{0}^{1,\overrightarrow{p}(\cdot)}(\Omega)$, such that
\begin{enumerate}
  \item $\displaystyle H(x,u, \nabla u)\in L^{1}(\Omega)$ and $\displaystyle \mid u\mid^{p_0(x)-2}u \in \L^1(\Omega),$
  \item for all $k>0$,\\
  $\displaystyle \int_{\Omega)}a(x,u,\nabla u)\nabla T_{k}(u-\xi)\,dx + \int_{\Omega}\Big(H(x,u,\nabla u)+|u|^{p_{0}(x)-2}u\Big)T_{k}(u-\xi)\,dx$\\
  $$\displaystyle \leq \int_{\Omega} f(x)T_{k}(u-\xi)\,dx, \quad \displaystyle \forall \xi \in \mathcal{C}_{0}^{1}(\Omega).$$.
\end{enumerate}
\end{definition}
The main result of the present work is the following theorem.
\begin{theorem}\label{TH1}
Let $\Omega \subset \RR^{N}$ be an unbounded domain, suppose that the assumptions (\ref{eq0}), (\ref{eqa1})--(\ref{eqh}) and $f\in L^{1}(\Omega)$ hold true. Then there exists at least one entropy solution of the problem $(\mathcal{P})$.
\end{theorem}
\begin{remark}
The Theorem above remains valid for $\xi \in W_{0}^{1,\overrightarrow{p}(\cdot)}(\Omega)\cap L^{\infty}(\Omega)$.
\end{remark}
\textbf{Proof of theorem \ref{TH1}}\\
Let $\displaystyle \Omega(n)=\{x\in \Omega: \quad |x|< n\}$, and $\displaystyle f^{n}(x)=\frac{f(x)}{1+\frac{|f(x)|}{n}}\chi_{\Omega(n)}$, we have
\begin{equation}
f^{n}\rightarrow f \quad \mbox{in}\quad L^{1}(\Omega), \quad n\rightarrow+\infty, \quad |f^{n}(x)|\leq |f|, \quad |f^{n}|\leq n\chi_{\Omega(n)}.
\end{equation}
$\displaystyle a^{n}(x,s,\xi)=\Big(a_{1}^{n}(x,s,\xi),...,a_{N}^{n}(x,s,\xi)\Big)$ where $\displaystyle a_{i}^{n}(x,s,\xi)=a_{i}(x,T_{n}(s),\xi)$, for $i=1,..,N$;\\
\begin{equation}\label{eqhn}
\displaystyle H^{n}(x,s,\xi)=T_{n}\Big(H(x,s,\xi)\Big)\chi_{\Omega(n)};\quad|H^{n}(x,s,\xi)|\leq H(x,s,\xi)\chi_{\Omega(n)}.
\end{equation}
Consider the following regularized equations:
\begin{equation}
    (\mathcal{P}_{n}) \left\{\begin{array}{l}
   \displaystyle \int_{\Omega}a(x,T_{n}(u^{n}),\nabla u^{n})\nabla v\,dx + \int_{\Omega}\Big(H(x,u^{n},\nabla u^{n})+|u^{n}|^{p_{0}(x)-2}u^{n}\Big).v \,dx\\
     \displaystyle = \int_{\Omega}f^{n}v \,dx\quad
   \mbox{for any} \quad  v \in  W_{0}^{1,\overrightarrow{p}(\cdot)}(\Omega).\\
    \end{array}%
    \right.
    \end{equation}
     The approximate problem ($\mathcal{P}_{n}$) admits a solution, this being based on the theorem of Lions \Big[\cite{Lions}Chapter II, 2, Theorem 2.7\Big], for pseudo-monotone operators. For the proof, see the Lemma \ref{l.app} in the Appendix.

  $\bf{Step~ 1: \mbox{A priori estimate of the sequence $\{u^m\}$}}.$ \\
  Let $ \displaystyle v=T_k(u^m)\exp(A(|u_m|))$ with $ \displaystyle A(s)=\int_0^s\frac{\hat{h}(r)}{\alpha}~dr.$ Since $v\in W_0^{1,\vec{p}}(\Omega)$, taking $v$ as the test function in the problem  $(\mathcal{P}_m)$, we get
 \begin{equation}\label{al1}
  \begin{aligned}
&\int_{\Omega}a^m(x,u^m,\nabla u^m)\nabla (T_k(u^m)\exp(A(|u_m|))) ~dx\\
&+\int_{\Omega}H^m(x,u^m,\nabla u^m) T_k(u^m)\exp(A(|u_m|))~ dx\\&+\int_{\Omega}|u^m|^{p_0(x)-2}u^m T_k(u^m)\exp(A(|u_m|))~dx\\&\leq\int_{\Omega}f_m(x)T_k(u^m)\exp(A(|u_m|)) ~dx.
 \end{aligned}
\end{equation}
 The first term in the left hand side can be written as
 \begin{equation}\label{al2}
  \begin{aligned}
&\int_{\Omega}a^m(x,u^m,\nabla u^m)\nabla (T_k(u^m)\exp(A(|u_m|))) ~dx\\
&=\int_{\Omega}a^m(x,u^m,\nabla u^m)\nabla T_k(u^m)\exp(A(|u_m|)) ~dx\\
&+\int_{\Omega}a^m(x,u^m,\nabla u^m)\nabla u^m|T_k(u^m)|\frac{\hat{h}(|u^m|)}{\alpha}\exp(A(|u_m|)) ~dx\\
&\geq\int_{\Omega}a^m(x,u^m,\nabla u^m)\nabla T_k(u^m)\exp(A(|u_m|)) ~dx\\&+\sum_{i=1}^N\int_{\Omega}|\partial_{x_i}u^m|^{p_i(x)}\hat{h}(|u^m|)|T_k(u^m)|\exp(A(|u_m|))~dx,
 \end{aligned}
 \end{equation}
 for the second term in the left hand side, by increasing condition for $H,$ we have
 \begin{equation}\label{al3}
  \begin{aligned}
\int_{\Omega}H^m(x,u^m,\nabla u^m) T_k(u^m)&\exp(A(|u_m|))~ dx\\ \leq& \sum_{i=1}^N\int_{\Omega}|\partial_{x_i}u^m|^{p_i(x)}\hat{h}(|u^m|)|T_k(u^m)|\exp(A(|u_m|))~dx\\&+
 \int_{\Omega}|h_0(x)||T_k(u^m)|\exp(A(|u_m|))~dx.
 \end{aligned}
 \end{equation}
 By combining (\ref{al1}), (\ref{al2}), and (\ref{al3}), we obtain
 \begin{equation}
  \begin{aligned}
 \sum_{i=1}^N\int_{\Omega}a_i(x,T_m(u^m),\nabla u^m)D^iT_k(u^m)&\exp(A(|u_m|)) ~dx\\&+\int_{\Omega}|u^m|^{p_0(x)-2}u^m T_k(u^m)\exp(A(|u_m|))~dx\\&\leq \int_{\Omega}|h_0(x)||T_k(u^m)|\exp(A(|u_m|))~dx\\&+
\int_{\Omega}f_m(x)T_k(u^m)\exp(A(|u_m|)) ~dx,
  \end{aligned}
 \end{equation}
which gives
\begin{equation}\label{es u^m}
\begin{aligned}
\sum_{i=1}^N\int_{\Omega}|D^iT_k(u^m)|^{p_i(x)}~dx&+\int_{\{|u^m|\leq k\}}|u^m|^{p_0(x)}\exp(A(|u_m|))~dx\\&+k\int_{\{|u^m|>k\}}|u^m|^{p_0(x)-1}\exp(A(|u_m|))~dx
\leq C_1k.
\end{aligned}
\end{equation}
$\bf{Step~ 2: \mbox{Almost everywhere convergence of sequence $\{u^m\}$}}$ \\
\begin{lemma}\label{lem1}
$\mbox{meas}\{x\in\Omega, ~~~~|u^m|>k\}$ tends to zeros as $k$ to infinity.
\end{lemma}
\begin{proof}
By (\ref{es u^m}), we have
$$\int_{|u^m|> k}|u^m|^{p_0(x)-1}~dx\leq C_2,
$$
which gives $\displaystyle \mbox{meas}\{x\in\Omega, ~~~~|u^m|>k\}k^{\bar{p}-1}\leq C_2$, for $k>1,$\\
then, \,\,$\displaystyle \mbox{meas}\{x\in\Omega, ~~~~|u^m|>k\}\longrightarrow0$ as $k\rightarrow\infty.$
\end{proof}
Let $\displaystyle g(k)=\sup_{m\in \NN}\mbox{meas}\{x\in\Omega, ~~~~|u^m|>k\}\rightarrow0$ as $k\rightarrow\infty.$
Since $\Omega$ is an unbounded domain in $\RR^{N},$ we define $\eta_R$ defined as
$$\eta_R(r)=\left\{\begin{array}{ccc}
  1 & \mbox{if} & r<R, \\
  R+1-r & \mbox{if} & R\leq r<R+1, \\
  0 & \mbox{if} & r\geq R+1.
\end{array}\right.$$
For $R, h>0$, we have
\begin{align*}
\sum_{i=1}^N\int_{\Omega}D^i(\eta_R(|x|)T_h(u^m))~dx\leq& C_2\sum_{i=1}^N\int_{\Omega}D^i(u^m\eta_R(|x|))~dx\\+&C_3\sum_{i=1}^N\int_{\Omega}T_h(u^m)D^i(\eta_R(|x|)~dx\leq C(h,R),
\end{align*}
which implies that the sequence $\{\eta_R(|x|)T_h(u^m)\}$ is bounded in\\
 $W_0^{1,\vec{p}(\cdot)}(\Omega(R+1))$, and by embedding Theorem, and since $\eta_R=1$ in $\Omega(R)$\\
we have
\begin{align*}
&\eta_RT_h(u^m)\longrightarrow v_h ~~\mbox{strongly in}~~~~ L^{\vec{p}(\cdot)}(\Omega((R+1))),\\
& T_h(u^m)\longrightarrow v_h ~~\mbox{strongly in}~~~~ L^{\bar{p}(\cdot)}(\Omega((R))).
\end{align*}
By Egorov's theorem, we can choose $E_h\subset \Omega(R)$ such that $\mbox{meas}(E_h)<\frac{1}{h}$ and $T_h(u^m)\longrightarrow v_h $ uniformly in $\displaystyle \Omega(R)\setminus E_h.$\\
Let $\Omega^h(R)=\{x\in\Omega(R)\setminus E_h:~~|v_h(x)|\geq h-1\}.$ Since $T_h(u^m)$ converges uniformly to $v_h$ in $\Omega(R)\setminus E_h,$ there exists $m_0$ such that for any $m\geq m_0$, $|T_h(u^m)|\geq h-2$ on $\Omega^h(R)$, i.e. $|u^m|\geq h-2,$  then by Lemma \ref{lem1} we obtain\\  $\displaystyle \mbox{meas}~\Omega^h(R)\leq \sup_{m}\mbox{meas}\{x\in\Omega:~~|u^m|\geq h-2\}=g(h-2)\longrightarrow0$ as $h$ tends to zeros.\\
Now we set $\Omega_h(R)=\{x\in\Omega(R)\setminus E_h:~~|v_h(x)|< h-1\},$ and remark that $\Omega(R)=\Omega^h(R)\cup\Omega_h(R)\cup E_h$ by combining the last results we have $\mbox{meas}~\Omega_h(R)>\mbox{meas}~\Omega(R)-1/h-g(h-2).$ The uniform convergence of $T_h(u^m)$ implies that, there exists $m_0\in\NN$ such that for $m\geq m_0, |T_h(u^m)|<h$ on $\Omega_h(R),$ which gives $u^m\longrightarrow v_h$ on $\Omega_h(R)$, by classical argument we can prove that $v_h$ does not depend on $h$, and the convergence
$$u_m\rightarrow u ~~\mbox{a.e. in}~~ \Omega(R), ~~ m\rightarrow\infty,$$
holds true. Then, by the diagonalisation argument with respect to $R \in \NN$, we establish the almost everywhere convergence of $u^m$ to $u.$\\
$\bf{Step~ 3: \mbox{Weak convergence of the gradient}}$ \\
By (\ref{es u^m}), we have
\begin{equation}
\|T_k(u^m)\|_{W_0^{1,\vec{p}(\cdot)}(\Omega)}\leq C(k).
\end{equation}
So we can extract a weakly convergent subsequence in $W_0^{1,\vec{p}(\cdot)}(\Omega),$ such that
\begin{align*}
&T_k(u^m)\rightharpoonup v_k ~~\mbox{weakly in}~~~~W_0^{1,\vec{p}(\cdot)}(\Omega),\\
&\nabla T_k(u^m)\rightharpoonup \nabla v_h ~~\mbox{weakly in}~~~~ L^{\vec{p}(\cdot)}(\Omega),\\
&T_k(u^m)\rightarrow v_k~~\mbox{strongly in}~~~~ L^{\vec{p}_0(\cdot)}(\Omega).
\end{align*}
Beside, $T_k(u^m)\rightarrow T_k(u)~~\mbox{a.e. in}~~~~ \Omega,$ gives
$T_k(u^m)\rightarrow T_k(u)~~\mbox{strongly in}~~~~ L^{\vec{p}_0(\cdot)}(\Omega),$ and we conclude that, $\displaystyle \nabla T_k(u^m)\rightharpoonup \nabla T_k(u) ~~\mbox{weakly in}~~~~ L^{\vec{p}(\cdot)}(\Omega).$\\
$\bf{Step~ 4: \mbox{Strong convergence of the gradient}}$ \\
Show that $\nabla T_k(u^m)\rightarrow \nabla T_k(u) ~~\mbox{in}~~~~ L_{loc}^{\vec{p}(\cdot)}(\Omega).$\\
Let $j>k>0,$ and
$$h_j(s)=\left\{\begin{array}{ccc}
  1 & \mbox{if} & |s|\leq j, \\
  1-|s-j| & \mbox{if} & j\leq |s|\leq j+1, \\
  0 & \mbox{if} & r> j+1.
\end{array}\right.$$
Taking $v=\exp(A(|u^m|))(T_k(u^m)-T_k(u))h_j(u^m)\eta_R(|x|)$ as test function in the approximate problem $(\mathcal{P}_m),$ we have
\begin{equation}\label{cvf1}
\begin{aligned}
&\sum_{i=1}^N\int_{\Omega}a_i(x,T_m(u^m),\nabla u^m)D^i(\exp(A(|u^m|))(T_k(u^m)-T_k(u))h_j(u^m)\eta_R(|x|)) ~dx\\
&+\int_{\Omega}H^m(x,u^m,\nabla u^m) \exp(A(|u^m|))(T_k(u^m)-T_k(u))h_j(u^m)\eta_R(|x|)~ dx\\&+\int_{\Omega}|u^m|^{p_0(x)-2}u^m \exp(A(|u^m|))(T_k(u^m)-T_k(u))h_j(u^m)\eta_R(|x|)~dx\\&\leq\int_{\Omega}f_m(x)\exp(A(|u^m|))(T_k(u^m)-T_k(u))h_j(u^m)\eta_R(|x|) ~dx.
\end{aligned}
\end{equation}
Denoting $J_1, J_2, J_3,$ and $J_4$ by
\begin{align*}
&J_1=\sum_{i=1}^N\int_{\Omega}a_i(x,T_m(u^m),\nabla u^m)\times\\
&D^i(\exp(A(|u^m|))(T_k(u^m)-T_k(u))h_j(u^m)\eta_R(|x|)) ~dx,\\
&J_2=\int_{\Omega}H^m(x,u^m,\nabla u^m) \exp(A(|u^m|))(T_k(u^m)-T_k(u))h_j(u^m)\eta_R(|x|)~ dx,\\
&J_3=\int_{\Omega}|u^m|^{p_0(x)-2}u^m \exp(A(|u^m|))(T_k(u^m)-T_k(u))h_j(u^m)\eta_R(|x|)~dx,\\
&J_4=\int_{\Omega}f_m(x)\exp(A(|u^m|))(T_k(u^m)-T_k(u))h_j(u^m)\eta_R(|x|) ~dx.
\end{align*}
$J_1$ can be rewritten as
\begin{align*}
&J_1=\\
&\sum_{i=1}^N\int_{\Omega}a_i(x,T_m(u^m),\nabla u^m)\times\frac{\hat{h}(u^m)}{\alpha}\times\\
&D^iu^m sign(u^m)\exp(A(|u^m|))(T_k(u^m)-T_k(u))h_j(u^m)\eta_R(|x|) ~dx\\
&+\\
&\sum_{i=1}^N\int_{\Omega}a_i(x,T_m(u^m),\nabla u^m)\exp(A(|u^m|))(D^iT_k(u^m)-D^iT_k(u))h_j(u^m)\eta_R(|x|)dx\\
&+\\
&\sum_{i=1}^N\int_{j\leq|u^m|\leq j+1}a_i(x,T_m(u^m),\nabla u^m)D^iu^m\exp(A(|u^m|))(T_k(u^m)-T_k(u))\eta_R(|x|)\\
&+\\
&\sum_{i=1}^N\int_{\Omega}a_i(x,T_m(u^m),\nabla u^m)\exp(A(|u^m|))(T_k(u^m)-T_k(u))h_j(u^m)D^i\eta_R(|x|))dx\\
&
=J_1^1+J_1^2+J_1^3+J_1^4.
\end{align*}
Since $h_j\geq0$ and $\eta_R(|x|)\geq0$ and $u^m(T_k(u^m)-T_k(u))\geq0,$ then
\begin{equation}\label{J_1^1-J_1^2}
\begin{aligned}
J_1^1\geq&\sum_{i=1}^N\int_{\Omega}\hat{h}(|u^m|)|D^iu^m|^{p_i(x)}\exp(A(|u^m|))|T_k(u^m)-T_k(u)|\eta_R(|x|)dx.\\
J_1^2\geq&\sum_{i=1}^N\int_{|u^m|\leq k}a_i(x,T_m(u^m),\nabla u^m)\exp(A(|u^m|))|D^iT_k(u^m)-D^iT_k(u)|\eta_R(|x|)dx\\
-&\sum_{i=1}^N\int_{k<|u^m|\leq j+1}|a_i(x,T_m(u^m),\nabla u^m)|\exp(A(|u^m|))|D^iT_k(u)|\eta_R(|x|)dx.
\end{aligned}
\end{equation}
\begin{equation}\label{J_1^3-J_2}
\begin{aligned}
J_1^3\geq&\sum_{i=1}^N\int_{j<|u^m|\leq j+1}a_i(x,u^m,\nabla u^m)D^iu^m|T_k(u^m)-T_k(u)|\eta_R(|x|) ~dx.\hspace{1.7cm}\\
J_2\leq&\sum_{i=1}^N\int_{|u^m|\leq k}\hat{h}(|u^m|)|D^iu^m|^{p_i(x)}\exp(A(|u^m|))|T_k(u^m)-T_k(u)|\eta_R(|x|) ~dx\\
+&\sum_{i=1}^N\int_{\Omega}|h_0(x)||T_k(u^m)-T_k(u)|\eta_R(|x|)\exp(A(|u^m|)) ~dx.
\end{aligned}
\end{equation}
By (\ref{cvf1}), (\ref{J_1^1-J_1^2}), (\ref{J_1^3-J_2}), $J_3$, and $J_4$, we have
\begin{equation}\label{es}
\begin{aligned}
&\sum_{i=1}^N\int_{|u^m|\leq k}a_i(x,T_k(u^m),\nabla T_k(u^m))\exp(A(|u^m|))|D^iT_k(u^m)-D^iT_k(u)|\eta_R(|x|)dx\\
&+\sum_{i=1}^N\int_{\Omega}a_i(x,T_m(u^m),\nabla u^m)\exp(A(|u^m|))(T_k(u^m)-T_k(u))\eta_R'(|x|) \,dx\\
& +\delta\sum_{i=1}^N\int_{|u^m|\leq k}|T_k(u_m)|^{p_0(x)-2}T_k(u_m)\exp(A(|u^m|))(T_k(u^m)-T_k(u))\eta_R(|x|)dx\\
&\leq\int_{\Omega}(|h_0(x)|+|f^m|)\exp(A(|u^m|))|T_k(u^m)-T_k(u)|\eta_R(|x|)\,dx\\
& +\sum_{i=1}^N\int_{k<|u^m|\leq j+1}a_i(x,T_{j+1}(u^m),\nabla T_{j+1}(u^m))\exp(A(|u^m|))|D^iT_k(u)|\eta_R(|x|)dx\\
&+\sum_{i=1}^N\int_{j<|u^m|\leq j+1}a_i(x,T_m(u^m),\nabla u^m)\times\\
& \hspace{3cm} D^iu^m\exp(A(|u^m|))|T_k(u^m)-T_k(u)|\eta_R(|x|)\,dx\cdot
\end{aligned}
\end{equation}
The first term in the right hand side goes to zeros as $m$ tends to $\infty,$
since $T_k(u^m)\rightharpoonup T_k(u)$ weak * in $L^{\infty}(\Omega).$
Since $(a_i(x,T_{j+1}u^m,\nabla T_{j+1}u^m))_m$ is bounded in $L^{p_i'(\cdot)}(\Omega_R),$ there exists
 $\xi_i\in L^{p_i'(\cdot)}(\Omega_R)$, such that
$|a_i(x,T_{j+1}u^m,\nabla T_{j+1}u^m)|\rightharpoonup \xi_i ~~\mbox{in} ~~ L^{p_i'(\cdot)}(\Omega_R),$ the second term of the left hand side tends to zeros.\\
The third term in the left hand side go to zeros.\\
 Indeed, by taking $v=\exp(A(|u^m|))T_1(u^m-T_j(u^m))\eta_R(|x|)$ as test function in the approximate problem $(\mathcal{P}_m)$, we obtain
$$\sum_{i=1}^N\int_{j<|u^m|\leq j+1}a_i(x,u^m,\nabla u^m)D^iu^m\eta_R(|x|) ~dx=0,
$$
and since $T_k(u^m)\rightharpoonup T_k(u)$ weak * in $L^{\infty}(\Omega)$, we conclude the result.\\
Since $T_k(u^m)\rightarrow T_k(u)$ strongly in $L_{loc}^{p_i(\cdot)}(\Omega),$ the second term of the right hand side increased by a quantity that tends to $0$ as $m$ tends to zero.\\
Finally (\ref{es}), rewrite as
\begin{equation}
\begin{aligned}
&\sum_{i=1}^N\int_{|u^m|\leq k}a_i(x,T_k(u^m),\nabla T_k(u^m))\exp(A(|u^m|))|D^iT_k(u^m)-D^iT_k(u)|\eta_R(|x|)dx\\&
+\delta\sum_{i=1}^N\int_{|u^m|\leq k}|T_k(u_m)|^{p_0(x)-2}T_k(u_m)\exp(A(|u^m|))(T_k(u^m)-T_k(u))\eta_R(|x|) ~dx\\&
\leq \varepsilon(j,m).
\end{aligned}
\end{equation}
then,
\begin{equation}\label{1}
\begin{aligned}
&\sum_{i=1}^N\int_{\Omega}(a_i(x,T_k(u^m),\nabla T_k(u^m))-a_i(x,T_k(u^m),\nabla T_k(u)))\times\\
&\exp(A(|u^m|))(D^iT_k(u^m)-D^iT_k(u))\eta_R(|x|) ~dx\\&
\leq-\sum_{i=1}^N\int_{\Omega}a_i(x,T_k(u^m),\nabla T_k(u))\exp(A(|u^m|))|D^iT_k(u^m)-D^iT_k(u)|\eta_R(|x|) ~dx\\&
-\sum_{i=1}^N\int_{|u^m|\leq k}a_i(x,T_k(u^m),\nabla T_k(u^m))\exp(A(|u^m|))D^iT_k(u)\eta_R(|x|) ~dx\\&
+ \varepsilon(j,m).
\end{aligned}
\end{equation}
In view of Lebesgue dominated convergence theorem, we have $T_k(u^m)\rightarrow T_k(u)$ strongly in $L_{loc}^{p_0(\cdot)}(\Omega)$ and
$D^iT_k(u^m)\rightharpoonup D^iT_k(u)$ weakly in $L^{p_i(\cdot)}(\Omega)$, then the terms on the right hand side of (\ref{1}) go to zeros as $m$ and $j$ tend  to infinity,\\
which gives  $$\displaystyle \sum_{i=1}^N\int_{\Omega(R)}\Big( a_i(x,T_k(u^m),\nabla T_k(u^m))-a_i(x,T_k(u^m),\nabla T_k(u))\Big)\times$$
\begin{equation}\label{2}
\Big( D^iT_k(u^m)-D^iT_k(u)\Big)\,dx \to 0.
\end{equation}
Thanks to Lemma \ref{cv gr}, we conclude that
\begin{equation}\label{a.e.cvg}
\nabla T_k(u^m)\longrightarrow\nabla T_k(u)~~~\mbox{a.e. in}~~~\Omega(R).
\end{equation}
As before, applying Egorov's theorem, we can find a set $E_k\subset \Omega(R),$ such that $\mbox{meas}\, E_k<1/k$ and $ T_k(u^m)\longrightarrow T_k(u)$ uniformly in $\Omega(R)\setminus E_k.$\\
Recall that $\mbox{meas} \Omega_k(R)>\mbox{meas} \Omega(R)-1/k-g(k-2)$, with \\
\begin{center}
$\Omega_k(R)=\{x\in\Omega(R)\setminus E_k:~~|u(x)|<k-1\},$
\end{center}

which gives  $\displaystyle |T_k(u^m)|<k$, on $\Omega_k(R)$, for any $m\geq m_0,$ \\
then
\begin{center}
$\nabla u^m\longrightarrow\nabla u~~~\mbox{a.e. in}~~~\Omega_k(R).$
\end{center}
Thus by the diagonalisation argument with respect to $R$, we obtain
\begin{equation}\label{a.e.cvg1}
\nabla u^m\longrightarrow\nabla u~~~\mbox{a.e. in}~~~\Omega,~~~\mbox{and} ~~~
\nabla T_k(u^m)\longrightarrow\nabla T_k(u)~~~\mbox{a.e. in}~~~\Omega.
\end{equation}
Since $H^m(x,u^m,\nabla u^m)\longrightarrow H(x,u,\nabla u)$ a.e. in $\Omega$, by Fatou's Lemma,
 \begin{center}
 $H(x,u,\nabla u)\in  L^1(\Omega).$
\end{center}
$\bf{Step~ 5: \mbox{Equi integrability of} ~~|u^m|^{p_0(x)-2}u^m ~~\mbox{and}~~ H^m(x,u^m,\nabla u^m)}$ \\
In this section we will prove that $H^m(x,u^m,\nabla u^m)\longrightarrow H(x,u,\nabla u)$ \\
and $|u^m|^{p_0(x)-2}u^m\longrightarrow|u|^{p_0(x)-2}u$
strongly in $L_{loc}^1(\Omega).$\\
 We have $$H^m(x,u^m,\nabla u^m)\longrightarrow H(x,u,\nabla u)\quad \mbox{a.e. in} \quad\Omega$$
 and
 \begin{center}
 $|u^m|^{p_0(x)-2}u^m\longrightarrow|u|^{p_0(x)-2}u$ a.e. in $\Omega$.
\end{center}
In view of Vitali's Theorem, it's sufficient to prove that $H^m(x,u^m,\nabla u^m)$ and $|u^m|^{p_0(x)-2}u^m$ are uniformly equi-integrable.\\
Let $\displaystyle v=\exp(2A(|u^m|))T_1(u^m)-T_h(u^m))$, remark that $v\in W_0^{1,\vec{p}(\cdot)}(\Omega)\cap L^{\infty}(\Omega),$ so taking $v$ as test function in the approximate problem $(\mathcal{P}_m)$, we have
\begin{equation}
\begin{aligned}
&\sum_{i=1}^N\int_{\Omega}a_i(x,T_m(u^m),\nabla u^m)D^i(\exp(2A(|u^m|))T_1(u^m-T_h(u^m))) ~dx\\&+\int_{\Omega}H^m(x,u^m,\nabla u^m) \exp(2A(|u^m|))T_1(u^m-T_h(u^m))~ dx\\&+\int_{\Omega}|u^m|^{p_0(x)-2}u^m \exp(2A(|u^m|))T_1(u^m-T_h(u^m))~dx\\&\leq\int_{\Omega}f_m(x)\exp(2A(|u^m|))T_1(u^m-T_h(u^m)) ~dx.
\end{aligned}
\end{equation}
By coercivity of $a_i$ and increasing conditions of $H^m$, we obtain
\begin{equation}
\begin{aligned}
&\sum_{i=1}^N\int_{\Omega}a_i(x,T_m(u^m),\nabla u^m)D^iu^m\frac{\hat{h}(|u^m|)}{\alpha}\exp(A(|u^m|))|T_1(u^m-T_h(u^m))|~dx\\&+\int_{k\leq|u^m|\leq k+1}a_i(x,u^m,\nabla u^m) \exp(2A(|u^m|))D^iu^m~ dx\\&+\delta\int_{\Omega}|u^m|^{p_0(x)-2}u^m \exp(2A(|u^m|))T_1(u^m-T_h(u^m))~dx\\&\leq C_1\int_{\Omega}(|f_m(x)|+|c(x)|)|T_1(u^m-T_h(u^m))| ~dx,
\end{aligned}
\end{equation}
it follows that
\begin{equation}
\begin{aligned}
\sum_{i=1}^N\int_{|u^m|>h+1}\hat{h}(|u^m|)|D^iu^m|^{p_i(x)}~dx&+\delta\int_{|u^m|>h+1}|u^m|^{p_0(x)-1}~dx\\&\leq C_1\int_{|u^m|>h}(|f(x)|+|c(x)|) ~dx.
\end{aligned}
\end{equation}
Thus, for $\varepsilon>0,$ there exists $h(\varepsilon)>0$ such that $\forall h>h(\varepsilon)$
\begin{equation}
\sum_{i=1}^N\int_{|u^m|>h+1}\hat{h}(|u^m|)|D^iu^m|^{p_i(x)}~dx+\delta\int_{|u^m|>h+1}|u^m|^{p_0(x)-1}~dx\leq \varepsilon/2.
\end{equation}
Let $Q$ be an arbitrary bounded subset for $\Omega$.\\
Then, for any measurable set $E\subset Q$, we have
\begin{equation}
\begin{aligned}
&\sum_{i=1}^N\int_{E}\hat{h}(|u^m|)|D^iu^m|^{p_i(x)}~dx+\delta\int_{E}|u^m|^{p_0(x)-1}~dx\\&\leq \sum_{i=1}^N\int_{E}\hat{h}(|T_{h+1}u^m|)|D^iT_{h+1}u^m|^{p_i(x)}~dx+\delta\int_{E}|T_{h+1}(u^m)|^{p_0(x)-1}~dx\\&+
\sum_{i=1}^N\int_{|u^m|>h+1}\hat{h}(|u^m|)|D^iu^m|^{p_i(x)}~dx+\delta\int_{|u^m|>h+1}|u^m|^{p_0(x)-1}~dx.
\end{aligned}
\end{equation}
We conclude that for all $E\subset Q$ with $\mbox{meas}(E)<\beta(\varepsilon)$, and $T_h(u^m)\longrightarrow T_h(u)$ in $W_0^{1,\vec{p}(\cdot)}(\Omega_R),$
\begin{equation}
 \sum_{i=1}^N\int_{E}\hat{h}(|T_{h+1}(u^m)|D^iT_{h+1}u^m|^{p_i(x)}~dx+\delta\int_{E}|T_{h+1}(u^m)|^{p_0(x)-1}~dx\leq \varepsilon/2.
\end{equation}
Finally, combining the last formulas, $\forall E\subset Q~~\mbox{such that}~~\mbox{meas}(E)\leq \beta(\varepsilon)$ we obtain
\begin{equation}
 \sum_{i=1}^N\int_{E}\hat{h}(|u^m||D^iu^m|^{p_i(x)}~dx+\delta\int_{E}|u^m|^{p_0(x)-1}~dx\leq \varepsilon,
\end{equation}
which gives the results.\\
$\bf{Step~5: \mbox{Passage to the limit}}$\\
Let $\displaystyle v=\psi_l T_k(u^m-\varphi)$, $\psi_l\in \mathcal{D}(\Omega)$ such that
$$\psi_l(x)=\left\{\begin{array}{ccc}
  1& \mbox{if}  & x\in\Omega(l) \\
  0 & \mbox{if} & x\in\Omega\setminus\Omega(l+1),
\end{array}\right.$$
 and taking $\varphi\in W_0^{1,\vec{p}(\cdot)}(\Omega)\cap L^{\infty}(\Omega)$ as test function in the approximate problem, we obtain
 \begin{equation}\label{p.l}
\begin{aligned}
  \sum_{i=1}^N&\int_{\Omega(l+1)}a_i(x,T_m(u^m),\nabla u^m)\psi_lD^iT_k(u^m-\varphi)~dx\\&+\sum_{i=1}^N\int_{\Omega(l+1)\setminus\Omega(l)}a_i(x,T_m(u^m),\nabla u^m)D^i\psi_lT_k(u^m-\varphi)~dx\\
  &+\int_{\Omega(l+1)}H^m(x,u^m,\nabla u^m)\psi_lT_k(u^m-\varphi)~dx\\
&+\int_{\Omega(l+1)}|u^m|^{p_0(x)-2}u^m\psi_lT_k(u^m-\varphi)~dx\\&=\int_{\Omega(l+1)}f_m\psi_lT_k(u^m-\varphi)~dx.
\end{aligned}
 \end{equation}
Let $\displaystyle M=k+\|\xi\|_{\infty}.$ If $\displaystyle |u^m|\geq M$, then $\displaystyle |u^m-\xi|\geq |u^m|-\|\xi\|_{\infty}\geq k.$ Therefore $\displaystyle \{|u^m-\xi|<k\}\subseteq \{|u^m|<M\}$, and hence
\begin{align*}
I_l^m&=\int_{\Omega}a_i(x,T_m(u^m),\nabla u^m)\psi_lD^iT_k(u^m-\varphi)~dx
\\&=\int_{\Omega}a_i(x,T_M(u^m),\nabla T_M(u^m))\psi_lD^iT_k(u^m-\varphi)~dx\\&=\int_{\Omega}a_i(x,T_M(u^m),\nabla T_M(u^m))\psi_l(D^iT_M(u^m)-D^i\varphi)\chi_{\{|u^m-\xi|<k\}}~dx,~~ m\geq M.
\end{align*}
Let $\displaystyle w^m = u^m-\varphi, w=u-\varphi$. We have
\begin{equation}
D^iT_k(w^m)-D^iT_k(w)=(D^iw^m-D^iw)\chi_{\{|y^m|<k\}}+D^iw(\chi_{\{|w^m|<k\}}-\chi_{\{|w|<k\}})\rightarrow0,
\end{equation}
 $\mbox{a.e. in}~ \Omega, m\rightarrow\infty$.
Using Young inequality  and the assumptions (\ref{eqa1}), (\ref{eqa3}), we deduce for any $\varepsilon \in  (0,1)$ that
$$a_i(x,T_M(u^m),\nabla T_M(u^m))(D^iT_M(u^m)-D^i\varphi)\chi_{\{|u^m-\xi|<k\}}\geq -c_1(|D^i\xi_i|^p_i(x)).$$
Since $-c_1(|D^i\xi_i|^p_i(x))\in L^1(\Omega)$ by Fatou's lemma we have
\begin{equation}
\begin{aligned}
\lim_{m\rightarrow\infty} \inf I_l^m\geq&\int_{\Omega}a_i(x,T_M(u),\nabla u)\psi_lD^iT_k(u-\varphi)~dx\\=&\int_{\Omega}a_i(x,u,\nabla u)\psi_l D^iT_k(u-\varphi)~dx.
\end{aligned}
\end{equation}
and
\begin{equation}
\psi_lT_k(u^m-\varphi)\rightharpoonup \psi_lT_k(u^m-\varphi)~~ \mbox{weakly * in}~~ L^{\infty}(\Omega),~~ m\rightarrow\infty.
\end{equation}
By equi-integrability of $H^m$ and $|u^m|^{p_0(x)-2}u^m$, passing to the limit on $m$ in (\ref{p.l}) we obtain
\begin{equation}
\begin{aligned}
  &\sum_{i=1}^N\int_{\Omega(l+1)}a_i(x,u,\nabla u)\psi_lD^iT_k(u-\varphi)~dx\\&+\sum_{i=1}^N\int_{\Omega(l+1)\setminus\Omega(l)}a_i(x,u,\nabla u)D^i\psi_lT_k(u-\varphi)~dx\\
  &+\int_{\Omega(l+1)}H(x,u,\nabla u)\psi_lT_k(u-\varphi)~dx
+\int_{\Omega(l+1)}|u|^{p_0(x)-2}u\psi_lT_k(u-\varphi)~dx\\&\leq\int_{\Omega(l+1)}f\psi_lT_k(u-\varphi)~dx.
\end{aligned}
 \end{equation}
Now passing to the limit to infinity in $l$ we obtain the existence of entropy solution for the problem.
\section{Appendix}
\begin{lemma}\label{l.app}
Let $\Omega \subset \RR^{N}$ be an unbounded domain, suppose that the assumptions (\ref{eq0}), (\ref{eqa1})--(\ref{eqh}), there exists at least one weak solution of the problem $(\mathcal{P}_{n})$.
\end{lemma}
\begin{proof}
Let $\Omega \subset \RR^{N}$ be an unbounded domain, for all $\displaystyle \forall u,v \in W_{0}^{1,\overrightarrow{p}(\cdot)}(\Omega)$, we denote by $A^{n}$ the operator defined from $\displaystyle W_{0}^{1,\overrightarrow{p}(\cdot)}(\Omega)$ into it's dual by:\\
$\displaystyle \langle A^{n}(u),v\rangle = \int_{\Omega}a(x,T_{n}(u),\nabla u)\nabla v\,dx + \int_{\Omega}\Big(H^{n}(x,u,\nabla u)+|u|^{p_{0}(x)-2}u\Big).v \,dx $\\
1)- We show that $A^{n}$ is bounded in $W_{0}^{1,\overrightarrow{p}(\cdot)}(\Omega)$:\\
Using (\ref{eqa1}) and the fact that $\displaystyle ||u||_{p'(\cdot)}\leq \bigg(\int_{\Omega}|u|^{p'(\cdot)}+1 \bigg)^{\frac{1}{p^{'-}}}$ we obtain the estimate
\begin{equation}\label{est1.1}
\begin{array}{lll}
\displaystyle ||a(x,T_{n}(u),\nabla u)||_{\overrightarrow{p'}(\cdot)} & = &\displaystyle \sum_{i=1}^{N}||a_{i}(x,T_{n}(u),\nabla u)||_{p'(\cdot)},\\
 &\leq& \displaystyle \sum_{i=1}^N\Big(\int_{\Omega}|a_{i}(x,T_{n}(u),\nabla u)|^{p'(x)}\,dx +1\Big)^{\frac{1}{p^{'-}}},\\
& \leq & \displaystyle \sum_{i=1}^N\Big(\hat{a}_{i}(n)\int_{\Omega}\sum_{i=1}^N|\frac{\partial u}{\partial x_{i}}|^{p_{i}(x)}+1\Big)^{\frac{1}{p^{'-}}},\\
& \leq & C_{1}(n,||\nabla u||_{\overrightarrow{p}(\cdot)}).\\
    \end{array}%
 \end{equation}
Likewise, (\ref{eqhn}) yields
\begin{equation}\label{est2}
\displaystyle ||H^{n}(x,u,\nabla u)||_{p'_{0}(\cdot)}\leq C_{2}(n)\quad \mbox{and}\quad  \displaystyle |||u|^{p_{0}(x)-2}u||_{p'_{0}(\cdot)}\leq C_{3}.
\end{equation}
The above estimates (\ref{est1.1}), (\ref{est2}) yield, for any $\displaystyle v \in W_{0}^{1,\overrightarrow{p}(\cdot)}(\Omega)$:\\
$$\displaystyle \langle A^{n}(u),v\rangle \leq C_{1}||\nabla v||_{\overrightarrow{p}(\cdot)}+2\Big(C_{2}(n) +C_{3}+||f^{n}||_{p'(\cdot)}\Big)(|| v||_{\overrightarrow{p}(\cdot)}+||\nabla v||_{p_{0}(\cdot)})$$
which gives the bounded of the operator $A^{n}$.\\
2)- $A^{n}$ is coercive:\\
In view of H\"{o}lder's type inequality, we have for all $u \in W_{0}^{1,\overrightarrow{p}(\cdot)}(\Omega).$
\begin{equation}\label{H.es}
\begin{aligned}
&\bigg|\int_{\Omega}H^{n}(x,u,\nabla u)u~dx\bigg|\\
&\leq\bigg(1/p_0^-+1/(p_0')^-\bigg)\bigg(\int_{\Omega}|H^{n}(x,u,\nabla u)|^{p_0'(x)}~dx+1\bigg)^{1/(p_0')^-}\|u\|_{p_0(\cdot)}\\
&\leq 2\bigg(n^{(p_0')^+}|\Omega(n)|+1\bigg)^{1/(p_0')^-}\|u\|_{1,\overrightarrow{p}(\cdot)}\\
&\leq C_n\|u\|_{1,\overrightarrow{p}(\cdot)}.
\end{aligned}
\end{equation}
Indeed, by means of (\ref{eqa3}), (\ref{eqh}), (\ref{H.es}) and  generalized Young inequality, we deduce that
\begin{equation}\label{est1}
\begin{array}{lll}
\displaystyle && \langle A^{n}(u),v\rangle\\& = & \displaystyle \int_{\Omega}a(x,T_{n}(u),\nabla u)\nabla u\,dx + \int_{\Omega}\Big(H^{n}(x,u,\nabla u)+|u|^{p_{0}(x)-2}u+ f^{n}\Big).u \,dx \\
&\geq &  \displaystyle \alpha \sum_{i=1}^N\int_{\Omega} |\frac{\partial u}{\partial x_{i}}|^{p_{i}(x)}\,dx -C_n||u||_{1,\overrightarrow{p}(\cdot)}+(1-\epsilon)\int_{\Omega}|u|^{p_{0}(x)}\,dx\\ &-& \displaystyle C_{\epsilon}\int_{\Omega}|f^{n}|^{p'_{0}(x)}\,dx\\
&\geq &\displaystyle  \min(\alpha, (1-\epsilon))\bigg(\sum_{i=1}^N ||\frac{\partial u}{\partial x_{i}}||_{p_{i}(\cdot)}^{p_{i}^{-}} +
||u||_{p_{0}(\cdot)}^{p_{0}^{-}}\bigg)-C_n||u||_{1,\overrightarrow{p}(\cdot)} -C'.\\
\end{array}%
 \end{equation}
In view of (\ref{eq0}) we have
$$\displaystyle \frac{\langle A^{n}(u),v\rangle}{||u||_{1,\overrightarrow{p}(\cdot)}} \geq \min(\alpha, (1-\epsilon))\bigg(\sum_{i=1}^N ||\frac{\partial u}{\partial x_{i}}||_{p_{i}(\cdot)} +||u||_{p_{0}(\cdot)}\bigg)^{\underline{p}-1}-C_n -\frac{C'}{||u||_{1,\overrightarrow{p}(\cdot)}}.$$
Thus, $\displaystyle \frac{\langle A^{n}(u),v\rangle}{||u||_{1,\overrightarrow{p}(\cdot)}}\to +\infty$ when $\displaystyle ||u||_{1,\overrightarrow{p}(\cdot)}\to +\infty.$\\
3)- We show that $A^{n}$ is pseudo-monotone operator:\\
Let us now prove that if
\begin{equation}\label{cv1}
u^j\rightharpoonup u ~~~~\mbox{in}~~W_0^{1,\vec{p}(\cdot)}(\Omega),
\end{equation}
\begin{equation}\label{cv2}
A^m(u^j)\rightharpoonup w ~~~~\mbox{in}~~W_0^{-1,\vec{p'}(\cdot)}(\Omega),
\end{equation}
\begin{equation}\label{cv3}
\big \langle A^m(u^j),u^j\big\rangle\leq \big \langle w,u\big\rangle,
\end{equation}
then
\begin{equation}
A^m(u)=w
\end{equation}
\begin{equation}
\lim_{j\rightarrow\infty}\big\langle A^m(u^j),u^j\big\rangle=\big \langle A^m(u),u\big\rangle.
\end{equation}
The convergence (\ref{cv1}) yields the estimate
\begin{equation}\label{bor u^j}
\|u^j\|_{W_0^{1,\vec{p}(\cdot)}(\Omega)}\leq C_1, ~~~~j\in\NN.
\end{equation}
Let us show the following convergence along a subsequence:
\begin{equation}\label{a.e.cv u_m}
u^j\rightarrow u ~~\mbox{a.e. in}~~\Omega,~~j\rightarrow\infty.
\end{equation}
For $R, h>0$, we have
\begin{align*}
\sum_{i=1}^N\int_{\Omega}D^i(\eta_R(|x|)T_h(u^m))~dx\leq& C_2\sum_{i=1}^N\int_{|u_m|\leq h}|D^iu^m|^{p_i(x)}~dx\\+&C_3\sum_{i=1}^N\int_{\Omega}|D^i(\eta_R(|x|)|^{p_i(x)}~dx\leq C(h,R),
\end{align*}
which implies that the sequence $\{\eta_R(|x|)T_h(u^m)\}$ is bounded in $W_0^{1,\vec{p}(\cdot)}(\Omega(R+1))$, and by embedding theorem, and since $\eta_R=1$ in $\Omega(R)$ we have
\begin{equation}\label{cv.f. u_n}
\begin{aligned}
&\eta_RT_h(u^m)\longrightarrow v_h ~~\mbox{in}~~~~ L^{\vec{p}(\cdot)}(\Omega((R+1))),\\
&T_h(u^m)\longrightarrow v_h ~~\mbox{in}~~~~ L^{\bar{p}(\cdot)}(\Omega((R))).
\end{aligned}
\end{equation}
\indent By Egorov's Theorem, we can choose $E_h\subset \Omega(R)$ such that $\displaystyle \mbox{meas}(E_h)<1/h$ \\
and   $\displaystyle \displaystyle T_h(u^m)\longrightarrow v_h $ uniformly in $ \Omega(R)\setminus E_h.$\\
Let $\Omega^h(R)=\{x\in\Omega(R)\setminus E_h:~~|v_h(x)|\geq h-1\}.$ Since $T_h(u^m)$ converges uniformly to $v_h$ in $\Omega(R)\setminus E_h,$ there exists $m_0$ such that for any $m\geq m_0, |T_h(u^m)|\geq h-2$ on $\Omega^h(R)$, i.e. $|u^m|\geq h-2,$  then by Lemma \ref{lem1} we obtain
  $$\mbox{meas}~\Omega^h(R)\leq \sup_{m}\mbox{meas}\{x\in\Omega:~~|u^m|\geq h-2\}=g(h-2)\longrightarrow0$$
   as $h$ tends to zeros.\\
\indent Now we set $\Omega_h(R)=\{x\in\Omega(R)\setminus E_h:~~|v_h(x)|< h-1\},$ and remark that\\
 $\Omega(R)=\Omega^h(R)\cup\Omega_h(R)\cup E_h$ by combining the last results we have\\
  $\displaystyle \mbox{meas}~\Omega_h(R)>\mbox{meas}~\Omega(R)-1/h-g(h-2).$ The uniform convergence of $T_h(u^m)$ implies that, there exists $\displaystyle m_0\in\NN$ such that for $\displaystyle m\geq m_0, |T_h(u^m)|<h$ on $\Omega_h(R),$ which gives $\displaystyle u^m\longrightarrow v_h$ on $\displaystyle \Omega_h(R)$, by classical argument we can prove that $v_h$ not depend on $h$, and the convergence
\begin{equation}\label{cv.p.p}
u^m\rightarrow u ~~\mbox{a.e. in}~~ \Omega(R), ~~ m\rightarrow\infty,
\end{equation}
by (\ref{cv.f. u_n}), we get
\begin{equation}\label{cv.f. u_n1}
u^m\longrightarrow u ~~\mbox{in}~~~~ L^{\bar{p}(\cdot)}(\Omega((R))),
\end{equation}
holds true. Then, by the diagonalisation argument with respect to $R \in \NN$, we obtain the result (\ref{a.e.cv u_m}).\\
From (\ref{bor u^j}), (\ref{est1}) we have the estimate
\begin{equation}
\|a(x,T_m(u^j),\nabla u^j)\|_{\vec{p'}(\cdot)}\leq C_2(m), ~~j\in \NN.
\end{equation}
Therefore, there exist functions $\tilde{a}^m\in L^{\vec{p'}(\cdot)}(\Omega)$ such that
\begin{equation}\label{cv fa}
a(x,T_m(u^j),\nabla u^j)\rightharpoonup \tilde{a}^m~~\mbox{in}~~L^{\vec{p'}(\cdot)}(\Omega), ~~j\rightarrow\infty.
\end{equation}
The estimate (\ref{est2}) implies the existence of function $\tilde{b}^m\in L^{p_0'(\cdot)}(\Omega)$ such that
\begin{equation}\label{cv fb}
H^m(x,u^j,\nabla u^j)\rightharpoonup \tilde{H}^m~~\mbox{in}~~L^{p_0'(\cdot)}(\Omega), ~~j\rightarrow\infty.
\end{equation}
Then, the estimates (\ref{est2}), (\ref{bor u^j}) yield
\begin{equation}
\||u^j|^{p_0(x)-2}u^j\|_{p_0'(\cdot)}\leq C_3, ~~j\in\NN.
\end{equation}
Thus, in view of the convergence (\ref{a.e.cv u_m}), we obtain from Lemma 4.3 that
\begin{equation}\label{cv fi}
|u^j|^{p_0(x)-2}u^j\rightharpoonup |u|^{p_0(x)-2}u ~~~\mbox{in}~~ L^{\vec{p'}(\cdot)}(\Omega), ~~j\rightarrow\infty.
\end{equation}
By (\ref{cv2}), and (\ref{cv fa})-(\ref{cv fi}), for any $v\in W_0^{1,\vec{p}(\cdot)}(\Omega)$ we deduce that
\begin{equation}\label{ps}
\begin{array}{cc}
&\big<w,v\big>=\lim_{j\rightarrow\infty}\big<A^m(u^j),v\big>=\lim_{j\rightarrow\infty}\big<a(x,T_m(u^j),\nabla u^j).\nabla v\big>+\\\\
&\lim_{j\rightarrow\infty}\big<(b^m(x,u^j,\nabla u^j)+|u^j|^{p_0(x)-2}u^j-f_m)v\big>=\\\\
&\big<\tilde{a}^m.\nabla v\big>+\big<\tilde{H}^m+|u|^{p_0(x)-2}u-f_m)v\big>.
\end{array}
\end{equation}
Evidently, the following equality is satisfied:
\begin{equation}
\begin{array}{cc}
&\big<A^m(u^j),u^j\big>=\big<a(x,T_m(u^j),\nabla u^j).\nabla u^j\big>+\\\\
&\big<(H^m(x,u^j,\nabla u^j)+|u^j|^{p_0(x)-2}u^j-f_m)u^j\big>
\end{array}
\end{equation}
(\ref{cv3}) and (\ref{ps}) give
\begin{equation}
\lim_{j\rightarrow\infty}\big<A^m(u^j),u^j\big>\leq\big<\tilde{a}^m.\nabla u\big>+\big<\tilde{H}^m+|u|^{p_0(x)-2}u-f_m)u\big>
\end{equation}
It follows from the convergence (\ref{cv1}) that
\begin{equation}
\lim_{j\rightarrow\infty}\big<f^m,u^j\big>=\big<f^m,u\big>.
\end{equation}
Then, from the inequality (\ref{eqh}) and the convergence (\ref{cv.f. u_n1}) we have
\begin{align*}
&\lim_{j\rightarrow\infty}|\big<H^m(x,u^j,\nabla u^j)(u^j-u)\big>|\leq m \lim_{j\rightarrow\infty}\int_{\Omega(m)}|u^j-u|~dx\\ &\leq C(m)\lim_{j\rightarrow\infty}\|u^j-u\|_{\bar{p}(\cdot),\Omega(m)}=0.
\end{align*}
Using this fact and the convergence (\ref{cv fb}), we conclude that
\begin{equation}
\lim_{j\rightarrow\infty}\big<H^m(x,u^j,\nabla u^j)u^j\big>=\big<\tilde{H}^m. u\big>,
\end{equation}
which gives
\begin{equation}\label{in1}
\lim_{j\rightarrow\infty}\sup\big<a(x,T_m(u^j),\nabla u^j).\nabla u^j+|u^j|^{p_0(x)-2}u^j\big>\leq\big<\tilde{a}^m.\nabla u+|u|^{p_0(x)}\big>.
\end{equation}
On the other hand, thanks to the assumption (\ref{eqa1}) , we have
\begin{equation}
\big \langle(a(x,T_m(u^j),\nabla u^j)-a(x,T_m(u^j),\nabla u)).\nabla (u^j-u)\big \rangle
\end{equation}
$$+\big \langle(|u^j|^{p_0(x)-2}u^j-|u|^{p_0(x)-2}u)(u^j-u)\big\rangle\geq0.$$
Then
\begin{equation}
\begin{aligned}
\big<(a(x,&T_m(u^j),\nabla u^j).\nabla u^j+|u^j|^{p_0(x)}\big>\\&\geq\big<a(x,T_m(u^j),\nabla u)).\nabla u\big>+\big<a(x,T_m(u^j),\nabla u).\nabla (u^j-u)\big>\\&+\big<(|u^j|^{p_0(x)-2}u^ju\big>+\big<|u|^{p_0(x)-2}u)(u^j-u)\big>.
\end{aligned}
\end{equation}
Using (\ref{cv.p.p}) and (\ref{eqa1}) we obtain that
\begin{equation}\label{cv ffa}
a(x,T_m(u^j),\nabla u)\rightarrow a(x,T_m(u^j),\nabla u)~~\mbox{strongly in}~~L^{\vec{p}'(\cdot)}(\Omega), j \rightarrow\infty.
\end{equation}
In view of the convergence (\ref{cv.p.p}), we deduce that
 \begin{equation}\label{in2}
\lim_{j\rightarrow\infty}\inf\big<a(x,T_m(u^j),\nabla u^j).\nabla u^j+|u^j|^{p_0(x)-2}u^j\big>\geq\big<\tilde{a}^m.\nabla u+|u|^{p_0(x)}\big>.
\end{equation}
Combining (\ref{in1}) and (\ref{in2}) we have
 \begin{equation}\label{eg}
\lim_{j\rightarrow\infty}\big<a(x,T_m(u^j),\nabla u^j).\nabla u^j+|u^j|^{p_0(x)-2}u^j\big>=\big<\tilde{a}^m.\nabla u+|u|^{p_0(x)}\big>.
\end{equation}
and
 \begin{equation}\label{A}
\lim_{j\rightarrow\infty}\big<A^m(u^j),u^j\big>=\big<w,u\big>,
\end{equation}
by combining (\ref{cv1}), (\ref{cv fa}), (\ref{cv fb}), (\ref{cv ffa})  and (\ref{eg}) we get
\begin{equation}
\big \langle(a(x,T_m(u^j),\nabla u^j)-a(x,T_m(u^j),\nabla u)).\nabla (u^j-u)\big\rangle
\end{equation}
$$+\big\langle(|u^j|^{p_0(x)-2}u^j-|u|^{p_0(x)-2}u)(u^j-u)\big\rangle=0.$$
By Lemma \ref{cv gr} we have
\begin{equation}
\begin{aligned}
&\nabla u^j\rightarrow \nabla u ~~\mbox{in}~~L^{\vec{p}(\cdot)}(\Omega),~~ j\rightarrow\infty,\\
&\nabla u^j\rightarrow \nabla u ~~\mbox{a.e. in}~~~~\Omega,~~ j\rightarrow\infty.
\end{aligned}
\end{equation}
Then Lemma \ref{cv faible}
 $$
\tilde{a}^m=a(x,T_m(u),\nabla u),~~~~ \tilde{H}^m=H^m(x,u,\nabla u).
$$
By (\ref{A}) we conclude the result, and the existence of solutions for the approximate problem is proved.
\end{proof}


%
%



\end{document}